\documentclass[10pt]{amsart}
\usepackage{latexsym, amssymb,amsmath,amsthm}
\usepackage[all, knot]{xy}
\xyoption{arc}
\renewcommand{\1} {\mathbf{1}}
\renewcommand{\L} {\Lambda}

\newcommand{\ol} {\overline}
\def \w{\omega}
\def \k{\Bbbk}
\def \a{\alpha}
\def \b{\beta}
\def \ep{\epsilon}

\def \bt{\bowtie}
\def \du{^{\ast}}
\def \bidu{^{\ast\ast}}
\def \op{^{\operatorname{op}}}
\def \cop{^{\operatorname{cop}}}

\def \BC{\mathbb{C}}
\def \BZ{\mathbb{Z}}
\def \BQ{\mathbb{Q}}
\def \BN{\mathbb{N}}
\def \Vect{\operatorname{Vect}}
\def \F{\mathcal{F}}
\def \FF{\mathcal{F}_\sigma}
\def \f{\mathbb{F}}
\def \onu{\overline{\nu}}
\def \FSexp{\operatorname{FSexp}}
\def \Eexp{\operatorname{exp}}
\def \ord{\operatorname{ord}}
\newtheorem*{main1}{Theorem A}
\newtheorem*{main2}{Theorem B}
\def \inv{{^{-1}}}
\def \Ind{\operatorname{Ind}}
\def \Res{\operatorname{Res}}
\def \DD{\mathcal{D}}
\def \Id{\operatorname{Id}}
\def\TY{\mathcal{TY}}
\newcommand\Jac[2]{{#1 \overwithdelims () #2}}
\newcommand\mtx[1]{{\left[\begin{array}{c|c} #1
\end{array}\right]}}

\def \dim{\operatorname{dim}}

\def \Hom{\operatorname{Hom}}

\def \H{\operatorname{H}}

\def \Rep{\operatorname{Rep}}
\def \sgn{\operatorname{sgn}}

\def \e{\varepsilon}

\def \C{\mathbb{C}}
\def \o{\otimes}
\def \ev{\textrm{ev}}
\def \db{\textrm{db}}
\def \CC{\mathcal{C}}
\def \RR{\mathcal{R}}
\def \ZC{\mathcal{Z(C)}}
\newcommand\ptr{\operatorname{\underline{ptr}}}

\newcommand\ptrl{\ptr^\ell}
\newcommand\ptrr{\ptr^r}

\def \id{\operatorname{id}}
\def \Tr{\textrm{Tr}}
\def \exp{\textrm{FSexp}}

\def \v{\bar{\nu}}
\def \Irr{\operatorname{Irr}}

\numberwithin{equation}{section}

\newtheorem{theorem}{Theorem}[section]
\newtheorem{lemma}[theorem]{Lemma}
\newtheorem{proposition}[theorem]{Proposition}
\newtheorem{corollary}[theorem]{Corollary}
\newtheorem{definition}[theorem]{Definition}
\newtheorem{example}[theorem]{Example}
\newtheorem{remark}[theorem]{Remark}

\newtheorem{conjecture}[theorem]{Conjecture}

\begin{document}
\title[On Total Frobenius-Schur Indicators]{On Total Frobenius-Schur Indicators}
 \thanks{The first author was partially supported by the NSF grant number 11071111 and the second author was partially supported by NSF grant number DMS1001566.}
\author{Gongxiang Liu}
\address{Department of Mathematics, Nanjing University, Nanjing 210093, China}
\email{gxliu@nju.edu.cn}
\author{Siu-Hung Ng}
\address{Department of Mathematics, Iowa State University, Ames, IA 50011, USA}
\email{rng@iastate.edu}
\curraddr{Department of Mathematics, Cornell University, Ithaca, NY 14850, USA}

\date{}
\maketitle

\begin{abstract}
We define total Frobenius-Schur indicator for each object in a spherical fusion category $\mathcal{C}$ as a certain canonical sum of its higher indicators. The total indicators are invariants of spherical fusion categories. If $\mathcal{C}$ is the representation category of a semisimple quasi-Hopf algebra $H$, we prove that the total indicators are non-negative integers which satisfy a certain divisibility condition. In addition, if $H$ is a Hopf algebra, then all the total indicators are positive. Consequently, the positivity of total indicators is a necessary condition for a quasi-Hopf algebra being gauge equivalent to a Hopf algebra. Certain twisted quantum doubles of finite groups and some examples of Tambara-Yamagami categories are discussed for the sufficiency of this positivity condition.
\end{abstract}

\section{Introduction}
The representation category $\Rep(H)$  of a Hopf algebra $H$ is certainly important to the understanding of the algebraic structure of $H$. The monoidal structure of $\Rep(H)$ has also been playing important roles in other areas of mathematics and physics. For instance, the quantum invariants of knots, links or 3-manifolds constructed from certain Hopf algebras are actually determined by the monoidal structures of their representation categories \cite{Turaev}.

Quasi-Hopf algebras are generalizations of Hopf algebras whose representation categories are also monoidal categories. Two quasi-Hopf algebras are said to be \emph{gauge equivalent} if their representation categories are monoidally equivalent. For any finite-dimensional quasi-Hopf algebra $K$ over $\BC$, one can obtain another quasi-Hopf algebra $K^F$ by twisting  $K$ with a gauge transformation $F \in K \o K$ \cite{Kassel}, but $K^F$ and $K$ are gauge equivalent. In general, two finite-dimensional quasi-Hopf algebras $K$ and $H$ are gauge equivalent if, and only if, there exists a gauge transformation $F \in K \o K$ such that $K^F$ and $H$ are isomorphic as quasi-bialgebras (cf. \cite{EG02, NS0}). However, it is generally difficult to decide the gauge equivalence of two finite-dimensional quasi-Hopf algebras if their Grothendieck rings happen to be isomorphic.

For any finite group $G$, and a normalized $3$-cocycle $\w$ on $G$ with values in $\BC^\times$, one can construct a semisimple quasi-Hopf algebra $D^\w(G)$, called a \emph{twisted quantum double} of $G$ \cite{DPR}. Dijkgraaf, Pasquier and Roche  have asked the question whether $D^\w(G)$ is gauge equivalent to some Hopf algebra in \cite{DPR}. One can even ask a more general question: How can one determine whether a given finite-dimensional quasi-Hopf algebra $K$ over $\BC$ is gauge equivalent to some Hopf algebra? In view of the reconstruction theorem,  the question is equivalent to ask for  the existence of a \emph{fibration}, i.e. a $\BC$-linear, faithful and  exact monoidal functor $\F: \Rep(K) \to \Vect_\BC$, where $\Vect_\BC$ denotes the category of finite-dimensional $\BC$-linear spaces.

This simply stated question is generally difficult to answer. One can work out the simplest example $D^\w(\BZ_2)$, where $\w$ is a non-trivial 3-cocycle of $\BZ_2$. It is not completely obvious that $D^\w(\BZ_2)$ is not gauge equivalent to any Hopf algebra \cite{MN1}.

The $n$-th Frobenius-Schur (FS) indicator $\nu_n(V)$ of a representation $V$ of a finite group was introduced for
more than a century. It has been generalized to the representations of a semisimple (quasi-)Hopf algebra \cite{LM, MN05}, to the primary fields of a rational conformal field theory \cite{Bantay97}, and more generally to the objects of a pivotal categories \cite{NS2}. These indicators are preserved by equivalences of pivotal category [loc. cit.]. The arithmetical properties of the FS indicators also encrypt the structure of the underlying tensor categories as well as the quasi-Hopf algebras. For instance, the indicators of any complex representation of a finite group are integers.

If $V$ is an object of a spherical fusion category $\CC$ over $\BC$, it has been shown that the sequence  $\nu(V):=\{\nu_n(V)\}_{n \in \BN}$ of FS indicators is a periodic sequence of cyclotomic integers. Moreover, there exists a global period $N$ for all the higher indicator sequences \cite{NS1}. This global period, denoted by $\FSexp(\CC)$, is called the Frobenius-Schur exponent of $\CC$, and it is also an invariant of spherical fusion categories. The representation category $\Rep(H)$ of a semisimple quasi-Hopf algebra over $\BC$ is a spherical fusion category \cite{ENO}, and we simply denote the Frobenius-Schur exponent of $\Rep(H)$ by $\FSexp(H)$.

The FS indicators of a representation of semisimple Hopf algebra are not necessarily integers (see \cite{KSZ} for example). However, if we consider the total indicator $\onu(V)$ defined for any object $V$ in a spherical fusion category $\CC$ by
$$
\onu(V):=\sum_{n=1}^N \nu_n(V)
$$
where $N=\FSexp(\CC)$, then we have the following integrality and divisibility theorem for quasi-Hopf algebras.

\begin{main1}
 Let $H$ be a semisimple quasi-Hopf algebra over $\BC$. For any finite-dimensional $H$-module $V$, $\onu(V)$ is a non-negative integer which satisfies the divisibility
 $$
\FSexp(H) \mid (\dim H) \cdot \onu(V)\,.
 $$
\end{main1}
In addition, for semisimple Hopf algebras, we have obtained the positivity of total indicators:
\begin{main2}
 Let $H$ be a semisimple Hopf algebra $\BC$.
 Then, $\v(V) \ge \frac{N \dim V}{\dim H}$  for all finite-dimensional $H$-modules $V$, where $N=\FSexp(H)$. In particular, $\v(V)>0$ for any non-zero $H$-module $V$.
\end{main2}
Since $\FSexp(H)=\Eexp(H)$ when $H$ is a semisimple Hopf algebra,  Theorem A provides another perspective for a conjecture of Kashina \cite{Kas99, Kas00}: $\Eexp(H)$ divides $\dim H$ for any semisimple Hopf algebra $H$ over $\BC$. Moreover, Theorem B yields a necessary condition for a semisimple quasi-Hopf algebra being gauge equivalent to a Hopf algebra. However, this necessary condition may not be sufficient in general. There are integral Tambara-Yamagami categories which has no fibration but their total indicators are all positive. Nevertheless,  the positivity of total indicators is a necessary and sufficient for abelian twisted quantum doubles being gauge equivalent to Hopf algebras.

The organization of the paper is as follows: In Section 2, we introduce the definition of  total Frobenius-Schur indicators and prove Theorem A. Section 3 is devoted to the proof of Theorem B. In Section 4, we consider twisted quantum
doubles of $D^\w(G)$ of a finite abelian group $G$, and show that if $D^{\w}(G)$ is commutative, then $D^{\w}(G)$ is gauge equivalent to a Hopf algebra if, and only if, $\v(V)>0$ for all $V \in \Rep(D^\w(G))$. This provides an answer to a question of Dijkgraaf, Pasquier and Roche \cite[p69]{DPR}. In Section 5, we compute the total indicators for some integral Tambara-Yamagami fusion categories, and we characterize those admitting a fibration in terms of total indicators. As a consequence, semisimple quasi-Hopf algebras with positive total indicators but not gauge equivalent to any Hopf algebra are found.

\section{Total Frobenius-Schur indicators}

Throughout this paper, unless stated otherwise,  we will work over the field $\C$ of complex numbers;
every monoidal category $\CC$ in this paper is assumed to be $\BC$-linear abelian with finite-dimensional Hom-spaces over
$\C$ and a \emph{strict} simple unit object $\1_\CC$. All (quasi-)Hopf algebras are assumed to be semisimple and finite-dimensional over $\BC$. We denote by Rep$(H)$ the $\C$-linear monoidal category of finite-dimensional representations of a quasi-Hopf algebra $H$. The unit object of $\Rep(H)$, simply denoted by $\1_H$, is the $H$-module $\BC$ equipped with the trivial $H$-action.

In this section, we collect some conventions, and recall some basic definitions and facts for the discussions in the remaining sections.
 The readers may refer to \cite{BK,ENO,M} for the basic theory of tensor categories and
\cite{LM,KSZ,NS2,NS0} for Frobenius-Schur indicators.  The aim of this section is to introduce the definition of total Frobenius-Schur indicators (abbr. total
 indicators), and to prove Theorem A.

 Let $\CC$ be a left rigid monoidal category with tensor product $\o$. The left dual of $V \in \CC$ is a triple $(V^*, \db, \ev)$ in which $V^* \in \CC$, and $\db: \1 \to V\o V\du$ and $\ev:V\du \o V \to \1$  are respectively the associated dual basis and evaluation morphisms of the left dual. The left duality on $\CC$ can be extended to a monoidal equivalence $(-)\du: \CC\op \to \CC$ and hence $(-)\bidu: \CC \to \CC$ defines a monoidal equivalence. A pivotal structure on $\CC$ is an isomorphism $j: \Id \to (-)\bidu$ of monoidal functors, and the pair $(\CC, j)$ is a called a \emph{pivotal category}. We will simply say that $\CC$ is a pivotal category when the pivotal structure is understood without ambiguity.

In a pivotal category $(\CC,j)$, one can define the left and
right pivotal traces for any endomorphism $f: V \to V$ in $\CC$ as
\begin{gather*}
\ptrr(f):=\left(\1\xrightarrow{\db}V\o
V^{\ast}\xrightarrow{f\o V^{\ast}}V\o V^{\ast}\xrightarrow{j_V\o V^{\ast}}V^{\ast\ast}\o
V^{\ast}\xrightarrow{\ev}\1\right),
\\
\ptrl(f):=\left(\1\xrightarrow{\db}V^{\ast}\o V^{\ast\ast}\xrightarrow{V^{\ast}\o
j_V^{-1}}V^{\ast}\o V\xrightarrow{V^{\ast}\o f}V^{\ast}\o
V\xrightarrow{\ev}\1\right)
\end{gather*}
respectively. Note that these traces are scalars as $\1$ is simple. A \emph{spherical category} is a pivotal category $\CC$ in which $\ptrr(f) = \ptrl(f)$ for all endomorphisms $f \in \CC$. In this case, we simply write $\ptr$ for the functions $\ptrr$ as well as $\ptrl$, and  $d(V):=\ptr(\id_V)$ is called the \emph{pivotal dimension} of $V$. In addition, if $\CC$ is semisimple with finitely many simple objects up to isomorphism, then $\CC$ is called a \emph{spherical fusion category} (cf. \cite{ENO} for more details on fusion categories). In this case, $d(V)$ is a non-zero real number (cf. \cite{ENO}), and
the \emph{global dimension} $\dim \CC$ of $\CC$ is defined as
$$
\dim \CC = \sum_{V \in \Irr(\CC)} d(V)^2
$$
where $\Irr(\CC)$ denotes a complete set of non-isomorphic simple objects of $\CC$.

By M\"uger \cite{M1}, the center $\ZC$ is a \emph{modular tensor category}. In particular, the associated \emph{twist} (or \emph{ribbon structure}) $\theta$ has finite order \cite{BK, Va}. Moreover, the forgetful functor $F: \ZC \to \CC$ admits a two-sided adjoint $K: \CC \to \ZC$.

Let $V$ be an object in a pivotal category $(\CC, j)$ with associativity isomorphism $\Phi$. We write $V^{\o n}$ for the $n$-fold tensor power of $V\in \CC$ with rightmost parentheses. By coherence theorem, there is a unique isomorphism
$$
\Phi^{(n)}: V^{\o (n-1)} \o V \to V^{\o n}
$$
 which is a composition of tensor products of $\id$, $\Phi$ and $\Phi\inv$. One can define the $\BC$-linear
 operator $E_V^{(n)}: \CC(\1, V^{\o n}) \to  \CC(\1, V^{\o n})$ by setting
\begin{multline*}
 E_V^{(n)}(f):=\bigg(\1\xrightarrow{\db} V\du \o V\bidu \xrightarrow{(V\du\o f)\o V\bidu}
 (V\du \o V^{\o n}) \o V\bidu \xrightarrow{\Phi\inv \o j\inv} \\ ((V\du \o V) \o V^{\o (n-1)}) \o V
 \xrightarrow{(\ev \o V^{\o (n-1)}) \o V}V^{\o (n-1)} \o V \xrightarrow{\Phi^{(n)}} V^{\o n} \bigg)\,.
\end{multline*}
Following \cite[Sect. 3]{NS2}, the \emph{$n$-th Frobenius-Schur indicator} $\nu_{n}(V)$ of $V$ is defined as the scalar
$$\nu_{n}(V)=\Tr(E_{V}^{(n)}).$$
These indicators are proved to be invariant under pivotal equivalences, and $\nu_n(V)$ is a cyclotomic integer in $\BQ(\zeta_n)$ where $\zeta_n$ is a primitive $n$-th root of unity (cf. \cite{NS2}).

Since the antipode of a semisimple Hopf algebra $H$ is an involution \cite{LR88}, the representation category $\Rep(H)$ is a spherical fusion category in which the pivotal structure is the usual canonical isomorphism $j_V : V\to V\bidu$ of finite-dimensional vector spaces. In this case, the pivotal dimension $d(V)$ of $V \in \Rep(H)$ is the ordinary dimension of $V$. More generally, for a semisimple quasi-Hopf algebra $H$, there is a unique (spherical) pivotal structure on $\Rep(H)$ such that $d(V)$ is the ordinary dimension of $V$ for all $V \in \Rep(H)$ \cite{ENO}. Moreover, $\nu_n(V)$ can be expressed in terms of the associator, the quasi-antipode and the normalized integral $\L$ of $H$ (cf. \cite[Sect. 4]{NS0}). When $H$ is a Hopf algebra, we recover the $n$-th Frobenius-Schur indicator formula of $V$ introduced in \cite{LM}:
\begin{equation}\label{eq:FS1}
\nu_n(V) = \chi_V(\Lambda^{[n]})
\end{equation}
where $\chi_V$ is the character of $H$ afforded by $V$, $\L$ is the normalized integral of $H$ and $\L^{[n]}=\L_1 \L_2 \cdots \L_n$. Here, $x_1 \o \cdots \o x_n$ denotes the Sweedler notation of the $n$-fold comultiplication of $x \in H$.

The \emph{Frobenius-Schur exponent}, denoted by $\FSexp(\CC)$,  of a spherical category $\CC$ (cf. \cite{NS1}) is  the least positive
integer $n$ such that $\nu_{n}(V)=d(V)$ for all $V \in \CC$. If such an integer does not exist, $\FSexp(\CC)$ is defined to be $\infty$. However, the Frobenius-Schur exponent of a spherical fusion category is always finite because the following theorem proved in \cite[Thm. 4.1 and 5.5]{NS1}.
\begin{theorem}\label{t1}
  Let $\CC$ be a spherical fusion category, $\theta$ the twist of $\ZC$ and  $K: \CC \to \ZC$ the two-sided adjoint of the forgetful functor $F: \ZC \to \CC$. Then, for $V \in \CC$,
  \begin{enumerate}
  \item[(i)] $\displaystyle \nu_n(V) = \frac{1}{\dim \CC} \ptr \left(\theta_{K(V)}\right)$ for all $V \in \CC$, and \medskip
  \item[(ii)] $\displaystyle\FSexp(\CC)=\ord(\theta)$\,. \qed
  \end{enumerate}
\end{theorem}

 For a simple object $V \in \CC$, it is known that $\nu_1(V) = 1$ if $V \cong \1$, and 0 otherwise. The second indicator of $V$ can only be $0,1, -1$ depending on whether $V$ is self-dual or not. By Theorem \ref{t1}, we know $\nu_N(V)=d(V)$ if  $N=\FSexp(\CC)$. The meaning of higher indicators of $V$ are more obscure and they are not rational integers in general (cf. \cite[Ex. 7.5]{KSZ}). The theorem implies that the indicator sequence $\nu(V):=\{\nu_n(V)\}_{n \in \BN}$ of $V$ is periodic for any object $V$ of a spherical fusion category $\CC$. Moreover, $\FSexp(\CC)$ is the global period of all the indicator sequences of $\CC$. The average value or the sum of these indicators over a period should also be an important invariant.
\begin{definition} Let $\CC$ be a spherical fusion category and $N=\FSexp(\CC)$. The total Frobenius-Schur indicator of $V\in \CC$, denoted by
$\v(V)$, is defined as
$$\v(V):=\sum_{i=1}^{N}\nu_{i}(V)\,.$$
\end{definition}

To prove Theorem A, we first derive a formula for the total indicator $\v(V)$ of an object $V$ in a spherical fusion category $\CC$ in terms of some data of the center $\ZC$. Recall from \cite[Proposition 8.1]{M1} that the forgetful functor $F:\ZC \to \CC$ has a two-sided adjoint $K: \CC\to \mathcal{Z(C)}$. It follows from the semisimplicity of $\CC$, we have
$$K(V)\cong \sum_{X\in \Irr(\ZC)}\dim(\ZC(K(V),X))\, X$$
for $V\in \CC$, where $\Irr(\mathcal{Z(C)})$ is a complete set of non-isomorphic simple objects of $\ZC$. For simplicity, we set
$$[U:V]_{\CC}:=\dim(\CC(U,V))\quad \text{ for all } U,V \in \CC\,.$$
Since $K$ is a left adjoint of $F$, $\ZC(K(V), X) \cong \CC(V, F(X))$ and so
$$
[K(V):X]_{\ZC} = [V: F(X)]_\CC
$$
for $X \in \ZC$ and $V \in \CC$.

Since $\dim (\CC(U,V)) = \dim (\CC(V,U)) = \dim (\CC(U\du ,V\du))$, we have
$$
[V:U]_\CC=[U:V]_\CC = [U\du, V\du]_\CC \text{ for all } U, V \in \CC.
$$
Thus, we have
\begin{equation}\label{eq:K}
K(V)\cong \sum_{X\in \Irr(\ZC)}[F(X):V]_\CC\, X\,.
\end{equation}

Let $\theta$ be the twist of $\ZC$. For any $X\in \Irr(\mathcal{Z(C)})$, we define $\omega_{X}\in \C$ by the equation
$$\theta_{X}=\w_{X}\id_{X}.$$
Note that $\w_{X}$ is an $N$-th root of unity by Theorem \ref{t1}.

\begin{proposition} \label{p1}
Let $\CC$ be a  spherical fusion categories over $\C$ with $\FSexp(\CC)=N$, and $V \in\CC$. Then
\begin{enumerate}
  \item[(i)] $\v(V)$ is an algebraic integer invariant under pivotal equivalence, i.e. if $\F: \CC\to \DD$ defines an equivalence of pivotal categories, then $\v(V) = \v(\F (V))$. \medskip
  \item[(ii)] Moreover,
\begin{equation}\label{eq:2.1}
  \v(V)=\frac{N}{\dim(\CC)}\sum_{\substack{X\in \Irr(\ZC)\\ \theta_X =\id_X}} [F(X):V]_{\CC}\,d(X)
\end{equation}
where $F:\ZC \to \CC$ is the forgetful functor.
\end{enumerate}
\end{proposition}
\begin{proof} Statement (i) is an  immediate consequence of the fact that $\nu_n(V)$ are algebraic integers for all $n$, and that both $\nu_n(V)$ and  $\FSexp(\CC)$ are invariant under pivotal equivalences.\smallskip\\
(ii)
By Theorem \ref{t1}, we have
$$\nu_{n}(V)=\frac{1}{\dim(\CC)}\ptr(\theta_{K(V)}^{n})=\frac{1}{\dim(\CC)}\sum_{X\in \Irr(\mathcal{Z(C)})}[F(X):V]_{\CC}\,\w_{X}^{n}\,d(X).$$
Therefore,
\begin{eqnarray*}
\v(V)&=&\sum_{i=1}^{N}\nu_{i}(V)\\
&=&\frac{1}{\dim(\CC)}\sum_{i=1}^{N}\sum_{X\in \Irr(\ZC)}[F(X):V]_{\CC}\, \w_{X}^{i}\, d(X)\\
&=&\frac{1}{\dim(\CC)}\sum_{X\in \Irr(\ZC)}[F(X):V]_{\CC}\sum_{i=1}^{N}\,\w_{X}^{i}\, d(X)\\
&=&\frac{N}{\dim(\CC)}\sum_{\substack{X\in \Irr(\ZC)\\ \w_{X}=1}}[F(X):V]_{\CC}\, d(X).
\end{eqnarray*}
Here, the last equality follows from the fact that $\omega_{X}$ is an $N$-th root of unity.
\end{proof}
We can now prove Theorem A.
\begin{proof}[Proof of Theorem A]
Let $H$ be a semisimple quasi-Hopf algebra. Consider the canonical  pivotal structure on $\Rep(H)$. Then  $\Rep(H)$ is a spherical fusion category with  $d(V)=\dim V$ for all $V \in \Rep(H)$. In particular, $\dim (\Rep(H)) = \dim H$ and $d(X)$ is a non-negative integer for all $X \in \mathcal{Z}(\Rep(H))$. It follows from Proposition \ref{p1} (ii) that
$$
\v(V)=\frac{N}{\dim H}\sum_{\substack{X\in \Irr(\ZC)\\ \w_{X}=1}}[F(X):V]_{\CC}\, d(X)
$$
is a non-negative rational number.  By Proposition \ref{p1} (i), $\v(V)$ is a non-negative integer. Since
$$
\frac{\v(V)\dim H}{N}=\sum_{\substack{X\in \Irr(\ZC)\\ \w_{X}=1}}[F(X):V]_{\CC}\, d(X) \in \BZ\,,
$$
we establish the divisibility  $N \mid (\dim H) \,\v(V)$.
\end{proof}

If $H$ is a semisimple Hopf algebra, $\Eexp(H) = \ord(\theta)=\FSexp(H)$  (cf. \cite[Thm 2.5]{EGexp99} and Theorem \ref{t1}). Therefore, Theorem A is related to the following well-known conjecture considered by Kashina \cite{Kas99, Kas00}.

\begin{conjecture}\label{conj}
 Let $H$ be a semisimple  Hopf algebra over  $\BC$. Then the exponent of $H$ divides $\dim (H)$.
\end{conjecture}

By the Cauchy theorem for Hopf algebras \cite[Sect. 6]{KSZ} (see so \cite[Thm. 8.4]{NS1}), $\dim H$ and $\Eexp(H)$ have the same prime factors. Thus, if $\gcd\{\v(V)\mid V \in \Irr(H)\}$ is relatively prime to $\dim H$, then the conjecture will be proved for $H$.  However, the Kac algebra $K$ of dimension $8$ is an example in which $\gcd\{\v(V)\mid V \in \Irr(K)\}=4$ (see Example \ref{eg0}). Nevertheless, any upper bound of $\gcd\{\v(V)\mid V \in \Irr(H)\}$ will shed light on the conjecture.

By  Theorem A, $\v(V)=0$ is an extreme value and it is possible for some quasi-Hopf algebra demonstrated in the following example.

\begin{example} \label{eg1}
{\rm
Let $G$ be a finite group and
$\omega$ a normalized 3-cocycle on $G$ with coefficients in
$\BC^\times$. Following \cite[Sect. 7]{NS0}, one can construct a quasi-Hopf algebra
$H(G,\omega)=(\C[G]\du, \Delta, \varepsilon, \phi, \alpha, \beta, S)$ where
multiplication, identity, comultiplication $\Delta$, counit $\e$, and antipode $S$
are the same as the structure maps of the dual $\C[G]\du$ of the group algebra $\BC[G]$,
and $\phi$, $\alpha$, and $\beta$ are given by
$$
\phi=\sum_{a,b,c \in G}\omega(a,b,c)e(a) \o e(b) \o e(c), \quad \alpha=1,
\quad\mbox{and} \quad \beta=\sum_{a\in G}\omega(a, a^{-1}, a)^{-1}e(a)\,,
$$
where $\{e(x)\mid x \in G\}$ is the dual basis of $G$ for
$\C[G]^{\ast}$. If $\w' \in Z^3(G, \BC^\times)$ is cohomologous to $\w$, then $H(G, \w)$ can be twisted to the quasi-bialgebra $H(G,\w')$ by a gauge transformation (see \cite[XV]{Kassel} for gauge equivalence of quasi-bialgebras).
In particular, if $\w$ is a coboundary of the 2-cochain $f: G \times G \to \BC^\times$, then
$H(G, \w)$ can be twisted by the gauge transformation $F=\sum_{a,b \in G} f(a,b)\, e(a) \o e(b)$ to   the ordinary bialgebra $\BC[G]\du$.

We now consider a special case when $G$ is of order 2. Let $G=\{1,x\}$ be an abelian group of order 2 and $\omega$ a 3-cocycle
of $G$ given by
$$
\omega(a,b,c)=\left\{\begin{array}{rl}
-1 &\mbox{if }a=b=c=x,\\
1 & \mbox{otherwise}.
\end{array}\right.
$$
Then $H=H(G,\omega)$ is a 2-dimensional commutative quasi-Hopf
algebra. Let $V$ be the nontrivial simple $H$-module. As computed in \cite[Ex. 5.4]{NS1}, $\exp(H)=4$ and\
$\nu_{n}(V)=\cos(\frac{n\pi}{2})$. Therefore, $\v(V)=0$.
}
\end{example}

\section{Proof of Theorem B}
Example \ref{eg1} shows that $\v(V) = 0$ for some simple module $V$ of a semisimple quasi-Hopf algebra. However, this cannot  happen for semisimple Hopf algebras which is stated in Theorem B. We will prove this theorem in this section.

Recall that the Drinfeld double $D(H)$ of a finite-dimensional Hopf algebra $H$ is the bicrossproduct ${H^*}\cop \bt H$  (cf. \cite{Mont93bk, Kassel}).  In particular, $D(H)$ is a finite-dimensional Hopf algebra with  comultiplication and multiplication of $D(H)$ described by
$$
\Delta(f \bt a) = (f_2 \bt a_1) \o (f_1 \bt a_2)
$$
and
$$
(f \bt a)(g \bt b) = f  (g(S\inv(a_3)?a_1)) \bt a_2 b
$$
respectively, for $f,g \in H^*$ and $a,b \in H$. Here $a_1 \o a_2$ and $f_1 \o f_2$  are respectively the Sweedler notation for the comultiplications of $a \in H$ and $f \in H^*$ with the summation suppressed. The Hopf algebra $D(H)$ admits a canonical universal $\RR$-matrix $R = \sum_i \e \bt h_i \o h^i \bt 1$ where $\{h_i\}_i$ and $\{h^i\}_i$ are dual bases of $H$ and $H^*$ respectively.

It should be well-known to experts that $\Rep(D(H))$ and $\Rep(D(H^*))$ are equivalent braided monoidal categories. For the sake of completeness, we provide a proof of the statement in terms of gauge equivalence for  subsequent discussion.

\begin{lemma}  \label{l1}
  Let $H$ be a finite-dimensional Hopf algebra over any field $\k$, not necessarily semisimple. Then $D(H)$ and $D(H^*)$ are gauge equivalent quasi-triangular Hopf algebras via  the gauge transformation $R \in D(H^*) \o D(H^*)$, which is the universal $\RR$-matrix of $D(H^*)$, and the algebra isomorphism
  $$\sigma:\;D(H) \to D(H^*),\;\;\;\; f \bt a \mapsto (1 \bt f)(a \bt \e).$$
  Here, $H^{**}$ is identified with $H$ as Hopf algebras under the natural isomorphism $j: H \to H^{**}$ of vector spaces.
\end{lemma}
\begin{proof}
Since $\Delta\op_{D(H^*)}(x) = R  \Delta_{D(H^*)}(x) R\inv$ for all $x \in D(H^*)$, it suffices to show that
$ \sigma:D(H) \to D(H^*)\cop$ is a bialgebra homomorphism, but this is straightforward verification.
\end{proof}

The algebra isomorphism $\sigma: D(H) \to D(H^*)$ defines a braided monoidal equivalence
$(\FF,\xi, \id): \Rep(D(H^*)) \to \Rep(D(H))$ (cf. \cite[XIII.3.2]{Kassel}). Here, $\FF: \Rep(D(H^*)) \to \Rep(D(H))$ is the $\BC$-linear functor with  $\FF(V)$ defined as the $D(H)$-module via the isomorphism $\sigma$ for $V \in \Rep(D(H^{\ast}))$, and $\FF:\Hom_{D(H^*)}(V, W) \to \Hom_{D(H)}(\FF V, \FF W)$ the identity function. The coherence isomorphism $\xi: \FF(V) \o \FF(W) \to \FF(V \o W)$ is the left action of $R\inv$, and  $\FF(\1_{D(H^*)}) = \1_{D(H)}$.

The Hopf algebra $H^*$ can be considered as a subalgebra of $D(H)$ and $D(H^*)$ via the embeddings $i_1: H^* \to D(H^*)$ and $i_2:H^* \to D(H)$ defined by
$$
i_1(f) = 1 \bt f, \quad i_2 (f) = f \bt 1
$$
respectively, for $f \in H^*$. It follows immediately from the definition of $\sigma$ that we have the commutative diagram of algebra maps:
\begin{equation}
\xymatrix{D(H) \ar[r]^-{\sigma} & D(H^*) \\
H^* \ar[u]^-{i_2} \ar[ru]_-{i_1} &
}  \,.
\end{equation}
This implies the following lemma for the two pairs of induction and restriction functors.
\begin{lemma}  \label{l2}
 Let $\Res_{H^*}^{D(H^*)}$ and $\Res_{H^*}^{D(H)}$ be the restriction functors along the embeddings $i_1$ and $i_2$ respectively, and $\Ind_{H^*}^{D(H^*)}$ and $\Ind_{H^*}^{D(H)}$ the associated induction functors. Then we have
\begin{equation}
  \FF \circ \Ind_{H^*}^{D(H^*)} = \Ind_{H^*}^{D(H)}\quad\text{and} \quad
   \Res_{H^*}^{D(H)} \circ \FF = \Res_{H^*}^{D(H^*)}\,. \qed
\end{equation}
\end{lemma}

The preceding lemmas hold for any finite-dimensional Hopf algebras. We now turn to semisimple Hopf algebras $H$. In this case, $\CC=\Rep(H)$ is a spherical fusion category, and the \emph{right} center $\ZC$ of $\CC$ is equivalent to the $\Rep(D(H))$ as braided monoidal category. Equipped with the canonical pivotal structure, $\Rep(D(H))$ is a modular tensor category with the twist $\theta$ given by the action of the Drinfeld element of $D(H)$. The forgetful functor $F: \ZC \to \CC$ is the restriction $\Res_{H}^{D(H)}$ which has a left adjoint $\Ind_{H}^{D(H)} :\CC \to \ZC$. Thus, by the uniqueness of adjoint functor, $K: \ZC \to \CC$ is equivalent to $\Ind_{H}^{D(H)}$. If $V$ is a self-dual $H$-module, then so is $K(V)$.  This observation even holds for spherical fusion categories.
\begin{lemma} \label{l3}
Let $\CC$ be a spherical fusion category, $\ZC$ the center of $\CC$ and $K: \CC \to \ZC$ the left adjoint of the forgetful functor $F: \ZC \to \CC$. Then $K(V)$ is a self-dual object of $\ZC$ whenever $V \in \CC$ is self-dual. In particular, $K(\1_\CC)$ is self-dual. Moreover, if $\theta$ is the twist of $\ZC$, then $\theta_{K(\1_\CC)} =\id_{K(\1_\CC)}$.
\end{lemma}
\begin{proof}
Note that for $X \in \Irr(\ZC)$, $X\du$ is isomorphic to a unique object of $\Irr(\ZC)$. Let $V \in \CC$ be self-dual. It follows from \eqref{eq:K} that
\begin{eqnarray*}
  K(V)^* & \cong & \sum_{X \in \Irr(\ZC)} [F(X):V]_\CC\, X^* \\
         & \cong & \sum_{X \in \Irr(\ZC)} [F(X\du): V]_\CC\, X \\
         & \cong & \sum_{X \in \Irr(\ZC)} [F(X)\du: V]_\CC\, X \\
         & \cong & \sum_{X \in \Irr(\ZC)} [F(X)\du: V\du]_\CC\, X \\
         & \cong & \sum_{X \in \Irr(\ZC)} [F(X):V]_\CC\, X \\
         & \cong & K(V)\,.
\end{eqnarray*}
Here, the third isomorphism is a consequence of the fact that the forgetful functor $F:\ZC \to \CC$ defines a monoidal functor, and the last isomorphism follows from the remark preceding Proposition \ref{p1}. Since $\1_\CC$ is self-dual, and so is $K(\1_\CC)$.

By \cite[Prop. 2.8 (iii)]{NS3},  $[F(X):\1_\CC]_\CC \ne 0$ implies $\w_X = 1$. Therefore,
$$
\theta_{K(\1_\CC)} =  \sum_{X \in \Irr(\ZC)} [F(X):\1_\CC]_\CC\, \theta_X = \sum_{X \in \Irr(\ZC)} [F(X):\1_\CC]_\CC\, \id_X =\id_{K(\1_\CC)}\,.
$$
This proves the last assertion.
\end{proof}
We can now prove Theorem B.
\begin{proof}[Proof of Theorem B] Let $\CC=\Rep(H)$. In view of Proposition \ref{p1} (ii), it suffices to show the inequality
$$
\sum_{\substack{X \in \Irr(D(H)) \\ \theta_X=\id_X}}[\Res_H^{D(H)} X: V]_\CC \dim (X) \ge \dim (V)
$$
for all   $V \in \CC$. This is equivalent to show that $[\Res_H^{D(H)} W: V]_\CC > 0$ for some  $D(H)$-module $W$ satisfying $\theta_W =\id_W$.

Recall that $H^*$ is also a semisimple Hopf algebra \cite{LR88}. By Lemma \ref{l3}, $Y = \Ind_{H^*}^{D(H^*)}(\1_{H^*})$ is a self-dual $D(H^*)$-module and $\theta_Y =\id_Y$. Moreover,
\begin{equation}
  Y \cong \sum_{U \in \Irr(D(H^*))} [U:\1_{H^*}]_{\CC'} U
\end{equation}
where $\CC'=\Rep(H^*)$.
 Since $\Rep(D(H^*))$ and $\Rep(D(H))$ are spherical categories equipped with their canonical pivotal structures, $\FF:\Rep(D(H^*))\to \Rep(D(H))$ defines an equivalence of braided monoidal categories as well as pivotal categories. Therefore, by \cite[Prop. 6.2]{GMN}, we have
$$
\id_{\FF Y} = \FF (\id_Y)= \FF (\theta_Y)  = \theta_{\FF Y}\,.
$$
In particular,  $W=\FF Y$  is a self-dual $D(H)$-module satisfying $\theta_W=\id_W$. Since $\Ind_H^{D(H)}$ is a left adjoint of $\Res_H^{D(H)}$,  we have
$$
[V: \Res_H^{D(H)}W]_\CC = [\Ind_{H}^{D(H)} V: W]_\ZC\,.
$$
 However, by Lemma \ref{l2}, $W = \Ind_{H^*}^{D(H)}\1_{H^*}$. Therefore,
\begin{eqnarray*}
  \Hom_{D(H)}\left(\Ind_{H}^{D(H)} V,\, W\right) & = & \Hom_{D(H)}\left(\Ind_{H}^{D(H)} V,\, \Ind_{H^*}^{D(H)}\1_{H^*}\right) \\
  & \cong & \Hom_{D(H)}\left(\left(\Ind_{H^*}^{D(H)}\1_{H^*}\right)\du \o \left(\Ind_{H}^{D(H)} V\right), \1_{D(H)}\right) \\
  & \cong & \Hom_{D(H)}\left(\left(\Ind_{H^*}^{D(H)}\1_{H^*}\right) \o \left(\Ind_{H}^{D(H)} V\right), \1_{D(H)}\right)\,.
\end{eqnarray*}
By \cite[Thm. 8 (1)]{Bu}, we have
$$
\left(\Ind_{H^*}^{D(H)}\1_{H^*}\right) \o \left(\Ind_{H}^{D(H)} V\right) \cong D(H)^{\dim V}
$$
as $D(H)$-modules. Thus,
$$
[\Ind_{H}^{D(H)} V: W]_\ZC = [D(H)^{\dim V}: \1_{D(H)}]_\ZC = \dim V,
$$
and so $[V: \Res_H^{D(H)} W]_\CC =\dim V$. This completes the proof.
\end{proof}

\begin{definition}
{\rm
 A quasi-Hopf algebra $H$ is said to be \emph{genuine}  if $H$ is not gauge equivalent to any ordinary Hopf algebra.
 }
\end{definition}
Theorem B provides a sufficient condition for a semisimple quasi-Hopf algebra being genuine.

\begin{corollary}  \label{c3.5}
Let $H$ be a semisimple quasi-Hopf algebra. If there exists $V\in \Rep(H)$ such that $\v(V)=0$, then $H$ is a genuine quasi-Hopf algebra.
\end{corollary}
\begin{proof} Assume contrary. Then $H$ is gauge equivalent to a Hopf algebra $H'$. There exists a monoidal equivalence $\F: \Rep(H) \to \Rep(H')$ (see \cite[Thm. 2.2]{NS0}). Thus, $\F$ preserves their canonical pivotal structures (cf. \cite{NS2}).
By Proposition \ref{p1} (i) and Theorem B, we have
$$
\v(V) = \v(\F(V)) > 0
$$
for all $V \in \Rep(H)$.
\end{proof}

The following example shows that the existence of vanishing total indicators can also be a necessary condition for a genuine quasi-Hopf algebra.

\begin{example} \label{eg0}
{\rm There are exactly four gauge inequivalent $8$-dimensional quasi-Hopf algebras with five simple objects $\{a_1, a_2, a_3, a_4, m\}$ and fusion rules given by
$$
\{a_1, a_2, a_3, a_4 \} \cong \BZ_2 \times \BZ_2,\quad m^2 \cong \sum_{i=1}^4 a_i \quad \text{and}\quad  m a_i \cong a_i m \cong m
$$
for $i=1,2,3,4$. Their categories of representations are Tambara-Yamagami categories \cite{TY}. They are
$\C[Q],\C[D]$, the $8$-dimensional Kac algebra $K$ and its twisted version $K_{u}$ (see \cite[Sec.6]{NS0} for details), where $Q$ and $D$ are, respectively,
the quaternion group and the dihedral group of order $8$. We can assume $a_1$ to be the unit object. Then $\nu_n(a_1)=1$ for all $n \ge 1$. The objects $a_2, a_3$ and $a_4$ are  1-dimensional representations of these algebras, and their orders are 2. The indicator $\nu_n(a_i) = 1$ if $n$ is even, and 0 otherwise, for $i=2,3,4$. Thus $\v(a_i)> 0$ for $i=1,2,3,4$. By \cite[Sec. 6]{NS0}, we find the following table:
$$
\begin{array}{|c|c|c|c|c|c|c|c|c|}
\hline
 & \nu_2(m)& \nu_3(m) & \nu_4(m)& \nu_5(m) & \nu_6(m)& \nu_7(m)& \nu_8(m) & \v(m)\\
 \hline
 K & 1 &  0 & 0 & 0 & 1 & 0 & 2 & 4\\
  \hline
 K_u & -1 &  0 & 0 & 0 & -1 & 0 & 2 & 0\\
  \hline
\C[D] & 1& 0& 2& 0& 1& 0& 2 & 3\\
\hline
\C[Q] & -1& 0& 2& 0& -1& 0& 2 & 1\\
\hline
\end{array}
$$
In particular, the Frobenius-Schur exponents of $K, K_{u}, \C[D], \C[Q]$ are $8,8,4,4$ respectively. The quasi-Hopf algebra $K_{u}$ is the only one in the list for which $\v(m)=0$.
Therefore, $K_{u}$ is  a genuine quasi-Hopf algebra. It is well-known that $K,\C[D],\C[Q]$ are all the noncommutative semisimple Hopf algebras of dimension $8$ (cf. \cite{Masuoka}). Their corresponding indicators $\nu_n(m)$ can be computed using the formula \eqref{eq:FS1}.}
\end{example}

\section{Twisted quantum doubles}
The main result of this section is to show that existence of vanishing total indicators is
 also a necessary condition for an \emph{abelian} twisted quantum double, considered in \cite{MN1}, being genuine.

 The {\em twisted quantum double} $D^{\omega}(G)$ of $G$ relative to a normalized $3$-cocycle
$\w: G \times G \times G \to \BC^\times$ is the semisimple quasi-Hopf algebra with underlying vector
space $\C[G]\du\o \C[G]$  in which multiplication,
comultiplication $\Delta$, associator $\phi$, counit $\varepsilon$ and quasi-antipode
$(S, \a, \b)$ are given by
\begin{eqnarray*}
&&(e(g) \o x)(e(h) \o y) =\theta_g(x,y) \delta_{g^x,h}
e(g)\o x y,\\
&&\Delta(e(g)\o x)  = \sum_{hk=g} \gamma_x(h,k) e(h)\o x
\o e(k) \o x,\\
&&\phi = \sum_{g,h,k \in G} \omega(g,h,k)^{-1} e(g) \o 1 \o
e(h) \o 1 \o e(k) \o 1,\\
&&\varepsilon(e(g)\o x) = \delta_{g,1}, \quad \alpha=1, \quad \beta=\sum_{g \in
G} \omega(g,g^{-1},g)e(g)\o 1,\\
&&S(e(g)\o x) =
\theta_{g^{-1}}(x,x^{-1})^{-1}\gamma_x(g,g^{-1})^{-1}e(x^{-1}g^{-1}x)\o
x^{-1},
\end{eqnarray*}
where $\{e(g)|g\in G\}$ is the dual basis of $\{g|g\in G\}$,  $\delta_{g,1}$ is the Kronecker delta,  $g^x=x^{-1}g x$, and
\begin{eqnarray*}
\theta_g(x,y) &=&\frac{\omega(g,x,y)\omega(x,y,(x y)^{-1}g x y)}{\omega(x,x^{-1}g x,y)}, \\
\gamma_g(x,y) & = & \frac{\omega(x,y,g)\omega(g, g^{-1}x g, g^{-1}yg)}{\omega(x,g,
g^{-1}y g)}
\end{eqnarray*}
for any $x, y, g \in G$ (cf. \cite{DPR, Kassel}).

As in \cite{MN1}, we denote by $\Gamma^{\omega}$ the set of all group-like elements of $D^{\omega}(G)$ and call $\omega$ \emph{abelian} if $D^{\omega}(G)$ is a commutative algebra. This can only happen when $G$ is abelian and $\Gamma^\w$ spans $D^{\omega}(G)$. The following theorem characterizes those abelian twisted doubles of finite groups which are genuine  quasi-Hopf algebras, and it provides an answer for a question of Dijkgraaf,  Pasquier and Roche posed in \cite[p69]{DPR}.

\begin{theorem}\label{t4.1} Let $G$ be a finite abelian group, and $\omega$ a normalized $3$-cocycle of $G$ such that $D^\w(G)$ is a commutative algebra. Then $D^{\omega}(G)$ is  a genuine quasi-Hopf algebra if, and only if, there exists
$V\in \Rep(D^{\omega}(G))$ such that $\v(V)=0$.
\end{theorem}

To prove this theorem, we need some preparations. Recall that
 an Eilenberg-MacLane (EM) 3-cocycle of $G$ is a pair $(\omega,c)$ where $\omega\in Z^{3}(G,\C^\times)$ and $c$ a $2$-cochain of $G$ satisfying
the following conditions:
\begin{equation}\label{eq:EM}
\frac{c(xy,z)}{c(x,z)c(y,z)} \frac{\omega(x,z,y)}{ \omega(x,y,z)\omega(z,x,y)}=1=
\frac{c(x,yz)}{c(x,y)c(x,z)}\frac{\omega(x,y,z)\omega(y,z,x)}{\omega(y,x,z)}
\end{equation}
for all $x,y,z \in G$.
The EM $3$-cocycle $(\omega,c)$ is a coboundary if there exists a $2$-cochain $h$ of $G$ such that
\begin{gather*} \omega=\delta h\quad\text{and}\quad
c(x,y)=\frac{h(x,y)}{h(y,x)}\,.
\end{gather*}
The EM cohomology group  $\H^{3}_{ab}(G, \C^{\times})$ is then defined by
$$
\H^{3}_{ab}(G, \C^\times) = \frac{Z_{ab}^3(G, \C^\times)}{B_{ab}^3(G, \C^\times)}
$$
where $Z_{ab}^3(G, \C^\times)$ and $B^{3}_{ab}(G,\C^\times)$ are respectively
the abelian groups of EM 3-cocycles and 3-coboundaries.

For $(\omega,c)\in Z^{3}_{ab}(G, \C^\times)$, one can assign the function $t(x):=c(x,x)$, called its trace. The trace $t$ of an EM 3-cocycle of $G$ is a quadratic function on $G$, that means
\begin{enumerate}
  \item[(i)] $t(x^a)=t(x)^{a^2}$ for $a\in \BZ$, and
  \item[(ii)]$b_{t}(x,y):=\frac{t(xy)}{t(x)t(y)}$ defines a bicharacter of $G$.
\end{enumerate}
The following theorem of Eilenberg and Mac Lane  is essential to the discussion in this section (cf.  \cite[Thm. 3]{Mac52} and  \cite[\S7]{JS}).
\begin{theorem}[Eilenberg-MacLane]\label{EM}
 Let $Q(G,\C^{\times})$ be the abelian group of all quadratic functions from $G$ to $\C^\times$. The map assigning to each 3-cocycle its trace
 induces an isomorphism
 $$\H^{3}_{ab}(G, \C^{\times})\xrightarrow{\cong} Q(G,\C^{\times}). \qed
 $$
 \end{theorem}

 \begin{remark}
 {\rm Recall the example of quasi-Hopf algebra $H(G,\omega)$ defined in Example \ref{eg1} for a normalized 3-cocycle $\w$ on a finite group $G$. Let $\CC=\Rep(H(G,\omega))$. Then $\Irr(\CC)$ consists of a set of 1-dimensional representations $V_x$ indexed by $G$. Moreover, a pair $(\w, c)$ is an EM 3-cocycle of a finite abelian group $G$ if, and only if,  $\Rep(H(G,\omega))$ is a braided spherical fusion category with the braiding given by
 $$
 c_{V_x, V_y}:= \left(V_x \o V_y \xrightarrow{c(x,y)} V_y \o V_x \right)
 $$
 for $x, y \in G$.
 }
   \end{remark}

For simplicity, we say that a quadratic function $t:G \to \BC^\times$ of a finite abelian group $G$ can be obtained from a bicharacter $b$ of $G$ if $t(x) = b(x,x)$ for all $x \in G$. We denote by $K(G, \BC^\times)$ the set of all quadratic functions which can be obtained from some bicharacters of $G$.

 \begin{lemma} \label{l4.1}
 Let $G$ be a finite abelian group and $(\omega,c)$ an EM $3$-cocycle of $G$ with trace $t$. Then:
 \begin{enumerate}
   \item[(i)] $\w$ is a coboundary if, and only if, $t \in K(G, \BC^\times)$. \smallskip
   \item[(ii)] $\omega^{2}$ is a $3$-coboundary of $G$.
 \end{enumerate}
 \end{lemma}
 \begin{proof}  Let us denote the cohomology class represented by the EM $3$-cocycle $(\omega,c)$ of $G$ as $[(\omega,c)]$. Then $\w$ is a coboundary if, and only if, $[(\w,c)] = [(1,b)]$ for some 2-cochain $b$ of $G$. By \eqref{eq:EM}, $b$ is a bicharacter of $G$ and $t(x)=c(x,x)=b(x,x)$ for all $x \in G$. Therefore, $t \in K(G, \BC^\times)$. Conversely, suppose there exists a bicharacter $b$ on $G$ such that $b(x,x) =t(x)$ for all $x \in G$. Since $(1,b)$ is an EM 3-cocycle which has the same trace as $(\w,c)$, $[(\w, c)]=[(1,b)]$ by Theorem \ref{EM}. In particular, $\w$ is a coboundary. This proves statement (i).\smallskip\\
 (ii) By  \eqref{eq:EM}, $b(x,y):=c(x,y)c(y,x)$ defines a bicharacter of $G$, and the trace of the EM 3-cocycle $(\w^2,c^2)$ can be obtained from $b$. It follows from (i) that $\w^2$ is a coboundary.
 \end{proof}

 \begin{lemma} \label{l4.2}
 Let $G$ be a finite abelian group and $(\w, c)$ an EM 3-cocycle of $G$. Then $\w$ is a coboundary if, and only if, $\v(V)>0$ for all simple $H(G,\w)$-module $V$. In this case,  $H(G,\w)$ is gauge equivalent to the Hopf algebra $\BC[G]^*$.
 \end{lemma}
 \begin{proof}
   In view of Example \ref{eg1}, if $\w$ is a coboundary, then $H(G, \w)$ can be twisted to the bialgebra $\BC[G]\du$ by a gauge transformation. It follows from Corollary \ref{c3.5} that $\v(V) >0$ for all simple $H(G, \w)$-modules $V$.
   Conversely, we assume the positivity of total indicators. Let $t$ be the trace of the EM 3-cocycle $(\w,c)$. By Lemma \ref{l4.1}(i), it is equivalent to show that $t \in K(G, \BC^\times)$. Since $G$ is a direct sum of its cyclic subgroups, by \cite[Lem. 6.2(i)]{MN1}, $t \in K(G, \BC^\times)$ if, and only if, $t_C \in K(C, \BC^\times)$  for each cyclic summand $C$ of $G$, where $t_C$ denotes the restriction of $t$ on $C$. Note that the restriction $(\w_C, c_C)$ of  $(\w,c)$ on $C$ is an EM 3-cocycle of $C$ and its trace is equal to $t_C$. Therefore, by Lemma \ref{l4.1}(i), it is enough to prove that $\w_C$ is a coboundary for each cyclic subgroup $C$ of $G$.

   By Lemma \ref{l4.1}(ii), $\w^2$ is a coboundary. We may simply assume $\w^2=1$ as $\v(V)$ are preserved by gauge equivalence of semisimple quasi-Hopf algebras. A complete set  of non-isomorphic simple $H(G,\w)$-modules is given by $\Irr(H(G, \w))=\{V_{x}|x\in G\}$. By \cite[Prop. 7.1]{NS0},
 \begin{equation} \label{eq:nun}
 \nu_{n}(V_{x})=\delta_{x^{n},1}\prod_{r=1}^{n-1}\omega(x,x^{r},x).
 \end{equation}
 The equation implies  that $\nu_n(V_x) = 0$ if $\ell \nmid n$ where $\ell = \ord(x)$. Since $\w^2=1$, $\w(x,x^{r},x)=\pm 1$ for all $r \in \BZ$. By \eqref{eq:nun}, $\nu_\ell(V_x) = \pm 1$ and
 $$
 \nu_{k\ell}(V_x) = \nu_\ell(V_x)^k\,.
 $$
 for all positive integer $k$. We claim that $\nu_\ell(V_x) = 1$.  Suppose not. Then  $\nu_\ell(V_x) = -1$.
Let $N=\FSexp(H(G,\w))$. Then $\nu_N(V_x) = \dim V_x=1$. Therefore, $N/\ell$ is even and so
$\v(V_x) = 0$ which contradicts the assumption of the positivity of total indicators.

We now have
$$
1 = \nu_\ell(V_x) = \prod_{r=1}^{\ell-1} \w(x,x^r,x)\,.
$$
Note that right hand side is also the $\ell$-th indicator of $V_x$ considered as an $H(C, \w_C)$-module where $C$ is the cyclic subgroup generated by $x$. Since $\nu_\ell(V_x)$ is invariant under gauge transformations, the product depends only on the cohomology class of $\w_C$ in $H^3(C, \BC^\times)$, which is a cyclic group of order $\ell$. A generating 3-cocycle $\phi$, as described in \cite{MS}, is defined by
 $$
 \phi(x^{i},x^{j},x^{k})=q^{\ol i(\ol j + \ol k - \ol{j+k})/\ell}
 $$
 where $\ol i$ denotes the least non-negative residue of $i$ modulo $\ell$, and $q$ a primitive $\ell$-th root of unity.  Thus, $\w_C$ is cohomologous to $\phi^a$ for some non-negative integer $0 \le a \le \ell-1$, and
 $$
 1 = \prod_{r=1}^{\ell-1} q^{a (r + 1 - \ol{r+1})/\ell} = q^a\,.
 $$
 This implies $a=0$ and so $\w_C$ is a coboundary. This completes the proof of the lemma.
 \end{proof}

Now we can proof Theorem \ref{t4.1}
 \begin{proof}[Proof Theorem \ref{t4.1}]
    By \cite[Cor. 3.6]{MN1}, $D^{\omega}(G)$ is spanned by the set $\Gamma^\w$ of group-like elements when $D^\w(G)$ is commutative.  In particular,
 $\Gamma^{\omega}$ is a finite abelian group which fits into an exact sequence of abelian groups
 $$
 1 \to \hat{G} \to \Gamma^\w \to G \to 1
 $$
 determined by $G$ and $\w$. Moreover,  $D^{\omega}(G)$ is isomorphic to $H(\Gamma^\w,\omega')$ as quasi-bialgebras, where $\w' \in Z^3(\Gamma^\w, \BC^\times)$ is the inflation of $\w\inv$ along the above map $\Gamma^\w \to G$ \cite[Sect. 9]{MN1}. The braiding of $\Rep(D^\w(G))$ determines an EM 3-cocycle $(\w', c)$ of $\Gamma^\w$. By Lemma \ref{l4.2}, the existence of a vanishing total indicator if, and only if, $\w'$ is a non-trivial  3-cocycle of $\Gamma^\w$. This is equivalent to that $H(\Gamma^\w,\omega')$, or equivalently $D^\w(G)$, is a genuine quasi-Hopf algebra.
 \end{proof}

\section{Tambara-Yamagami categories}
In this section, we study the positivity of total indicators for \emph{integral} Tambara-Yamagami categories associated with elementary $p$-groups \cite{TY}. By \cite{ENO}, these fusion categories are monoidally equivalent to the categories of representations of semisimple quasi-Hopf algebras. In \cite{Ta}, a necessary and sufficient condition is obtained for an  integral Tambara-Yamagami category admitting a fibration, i.e. it is monoidally equivalent to the category of representations of a semisimple Hopf algebra. By Corollary \ref{c3.5}, positivity of total indicators is a  necessary condition for the existence of a fibration, but not sufficient in general. We will demonstrate the sufficiency of positivity of total indicators for these integral Tambara-Yamagami categories in this section.

Let $A$ be a finite abelian group. Tambara and Yamagami \cite{TY} classified fusion categories with a complete set of simple objects $A\sqcup \{m\}$ satisfying the fusion rules
\begin{equation}a\o b\cong ab,\;\;a\o m\cong m\cong m\o a,\;\;m\o m\cong \bigoplus_{x\in A}x \quad(a,b\in A).\end{equation}
They also showed that such categories are parametrized by pairs $(\chi,\tau)$ where $\chi$ is a non-degenerate symmetric bicharacter
of $A$, and  $\tau$ is a square root of $|A|^{-1}$. The corresponding fusion category, denoted by $\mathcal{TY}(A,\chi,\tau)$, can be described more precisely in the following definition.
\begin{definition} {\rm $\mathcal{TY}(A,\chi,\tau)$ is a skeletal fusion category over $\BC$ with the set of simple objects $S:=A\sqcup \{m\}$. $\Hom$-sets between
elements of $S$ are given by $$\Hom(s,s')=\left \{
\begin{array}{ll} \C & \;\;\;\;\textrm{ if }s=s',\\
0 &
\;\;\;\;\textrm{otherwise}
\end{array}\right. $$
with $\id_s = 1 \in \BC$.
The compositions of morphisms are obvious one. Tensor products of elements of $S$ are given by (5.1) (but with $\cong$ replaced by $=$). The unit object $\1$ is strict and equal to the identity $e\in A$. The
associativity constraint $\Phi$ is determined by
\begin{eqnarray*}
&&\Phi_{a,m,b}=\chi(a,b)\id_{m}:\;m\to m,\\
&&\Phi_{m,a,m}=(\chi(a,x)\delta_{x,y}\id_{x})_{x,y}:\;\bigoplus_{x\in A}x\to \bigoplus_{y\in A}y,\\
&&\Phi_{m,m,m}=(\tau\chi(x,y)^{-1}\id_{m})_{x,y}:\;\bigoplus_{x\in A}m\to \bigoplus_{y\in A}m,
\end{eqnarray*}
where $a,b\in A$, and the other $\Phi_{s,t,u}\;(s,t,u\in S)$ are identity morphisms. Duality of $\mathcal{TY}(A,\chi,\tau)$ is described as follows: For $a\in A$, $a^{\ast}:=a^{-1}$ with $\ev_a,\db_a$ being identity $\id_e$. For the object $m$, $m^{\ast}:=m$ with morphisms $\ev_m=\tau^{-1}\pi: m\otimes m\to \1$ and $\db_m=\iota: \1\to m\otimes m$ where $\pi: m\otimes m\to \1$ and $\iota: \1\to m\otimes m$ are respectively the canonical projection and embedding satisfying $\pi \iota = \id_\1$
 }
\end{definition}

Since the Frobenius-Perron dimension of $\TY(A,\chi, \tau)$ is an integer, $\TY(A, \chi, \tau)$  is pseudo-unitary by \cite{ENO}. In particular, there exists a unique spherical pivotal structure $j$ for which the pivotal dimension $d(V)$ of an object $V$ is equal to its Frobenius-Perron dimension, i.e.
$$
d(a) = 1 \quad \text{for}\quad a \in A, \quad \text{and}\quad d(m) = \sqrt{|A|}\,.
$$
 On can verify directly that $j_a = \id_a$ for $a \in A$, and $j_{m}=\sgn(\tau)\id_{m}$, where $\sgn(\tau)$ means the sign of the real number $\tau$. We will always assume $\TY(A,\chi, \tau)$ is a spherical fusion category relative to this canonical pivotal structure $j$.

By \cite{TY}, two Tambara-Yamagami categories $\TY(A,\chi,\tau)$ and $\TY(A',\chi',\tau')$ are monoidally equivalent if, and only if,  $\tau=\tau'$ and
there exists a group isomorphism $\sigma:A\to A'$ satisfying $\chi(a,b)=\chi'(\sigma(a),\sigma(b))$ for all $a,b\in A$ (i.e. $(A,\chi)$ and $(A',\chi')$ are  \emph{isometric}).

We  now turn to the case of elementary $p$-groups $V$ of order $p^r$, where $p$ is a prime. Then $V$ is  an $r$-dimensional vector space over the finite field $\f_p$ of order $p$.  A non-degenerate symmetric bilinear form $B: V \times V \to \f_p$ on $V$ determines a non-degenerate symmetric bicharacter $\chi_B$ defined by $\chi_B(a,b) = \Eexp(2\pi i B(a,b))$; moreover, the  assignment $B \to \chi_B$ is a one-to-one correspondence. Two bilinear forms $B$ and $B'$ respectively on the $\f_p$-spaces $V$ and $V'$ are said to be \emph{isometric} if there exists an isomorphism $\sigma : V \to V'$ of $\f_p$-spaces such that $B'(\sigma(a), \sigma(b))=B(a,b)$ for all $a, b \in V_1$. It is easy to see that
two bilinear forms $B, B'$ are isometric if, and only if, $\chi_{B}$, $\chi_{B'}$ are isometric bicharacters. Moreover, any bilinear form $B$ on a $\f_p$-space is uniquely determined by its \emph{Gram matrix} $[B(v_i, v_j)]_{ij}$ relative to a basis $\{v_i\}_i$ of $V$. In particular, we will denote by $B_0$ the bilinear form on $\f_p^r$ whose Gram matrix relative to the standard basis is the identity

\begin{remark}\label{r:5.1}
{\rm
  The Tambara-Yamagami category $\TY(A, \chi, \tau)$ is integral if, and only if, $d(m)=\sqrt{|A|}$ is an integer, or equivalently $|A|$ is a square. In this case, by \cite{ENO}, the fusion category is monoidally equivalent the representation category of a semisimple quasi-Hopf algebra over $\BC$. If $V$ is an elementary $p$-group of order $p^r$, then
  $\TY(V, \chi, \tau)$ is integral if, and only if, $r$ is even.
  }
\end{remark}

\subsection{Characteristic two.} We will show in this subsection that a Tambara-Yamagami category associated with an elementary 2-group of square order admits a fibration if, and only if, all its total indicators are positive.

Recall that all the non-degenerate alternating bilinear forms on $V=\f_2^r$ with $r$ even are isometric. Any non-degenerate symmetric bilinear form on $V$, which is not alternating, is isometric to  $B_0$. In particular, there are exactly two isometric classes of non-degenerate symmetric bilinear on $\f_2^r$ when $r$ is even. For odd $r$, every non-degenerate symmetric bilinear form is isometric to $B_0$ (see, for example, \cite{AA}).

Using this classification of symmetric bilinear forms, the indicators of the object $m$ in these Tambara-Yamagami categories have been obtained by Shimizu \cite[Thm. 6.3]{Sh}:  For any non-degenerate symmetric bilinear form $B$ on $\f_2^r$, the $n$-th indicator $\nu_n(m)$ of the simple object $m$ in $\TY(\f_2^r, \chi_B, \tau)$ is zero if $n$ odd. Moreover,
\begin{enumerate}
  \item[(i)] if $B$ is not alternating, then
  \begin{equation}\label{eq:5.4}
     \nu_{2k}(m)=\sgn(\tau)^k
\left(\frac{1+i}{\sqrt{2}}\right)^{rk}\left(\frac{1+i^{-k}}{\sqrt{2}}\right)^r\,.
  \end{equation}

  \item[(ii)] If $r$ is even and $B$ is alternating, then
  \begin{equation}\label{eq:5.5}
  \nu_{2k}(m)=\left \{
\begin{array}{ll} \sgn(\tau) & \text{if $k$ is odd,}\smallskip \\
  2^{r/2} & \text{if $k$ is even.}
\end{array}\right.
 \end{equation}
\end{enumerate}
We can now compute the total indicators for the integral Tambara-Yamagami categories associated with an elementary 2-group of square order.
\begin{proposition} \label{p:5.6}
Let $V=\f_{2}^{2\ell}$, $B$ a non-degenerate symmetric bilinear form on  $V$, $\tau$ a square root of $|V|\inv$, and $\CC=\TY(V,\chi_B,\tau)$.
\begin{enumerate}
  \item[(i)] If $B$ is not alternating, then
$\FSexp(\CC)=8$ and $\v(m)=2 \sgn(\tau)+2^{\ell}$. \smallskip
\item[(ii)] If $B$ is alternating, then
$\FSexp(\CC)=4$ and $\v(m)=\sgn(\tau)+2^{\ell}$.
\end{enumerate}
In particular, $\v(m)=0$ if, and only if, $B$ is not alternating, $\ell=1$ and $\sgn(\tau)=-1$.
\end{proposition}
\begin{proof}
  Note that for all $a \in V$, $\nu_n(a)=1=d(a)$ if $n$ is even, and 0 otherwise. \smallskip \\
  (i) If $B$ is not alternating, then, by \eqref{eq:5.4}, we have
$$\nu_{2}(m)=\sgn(\tau),\quad \nu_{4}(m)=0, \quad \nu_{6}(m)=\sgn(\tau),\;\;
\nu_{8}(m)=2^{\ell}=d(m).$$
Therefore, $\FSexp(\CC)=8$ and $\bar{\nu}(m)=\sum_{k=1}^{4}\nu_{2k}(m)=2\sgn(\tau)+2^{\ell}$. \smallskip \\
(i) If $B$ is alternating, then, by \eqref{eq:5.5}, we have
$$\nu_{2}(m)=\sgn(\tau) \quad\text{and}\quad
\nu_{4}(m)=2^{\ell}=d(m).$$
Therefore, $\FSexp(\CC)=4$ and $\v(m)=\sum_{k=1}^{2}\nu_{2k}(m)=\sgn(\tau)+2^{\ell}$.

The last statement follows directly from statement (i) and (ii).
\end{proof}

\begin{corollary}
Let $\CC$ be a Tambara-Yamagami category associated with an elementary 2-group $V$ of order $2^{2\ell}$. Then $\CC$  is monoidally equivalent to $\Rep(H)$ for some  semisimple Hopf algebra $H$ if, and only if, all its total indicators are positive.
\end{corollary}
\begin{proof}
By \cite[Prop 5.5]{Ta}, $\CC$ has a fibration if, and only if, $\CC$ is not monoidally equivalent to $\TY(\f_2^2, B_0, -2)$. It follows from Proposition \ref{p:5.6} that the simple object $m$ in $\CC$ satisfies $\v(m)>0$ if, and only if, $\CC$ is not monoidally equivalent to $\TY(\f_2^2, B_0, -2)$. This proves the corollary.
\end{proof}
\subsection{Odd characteristic} In contrast to the characteristic two case, positivity of total indicators may not be sufficient for the existence of fibration of an integral Tambara-Yamagami category associated with an elementary $p$-group for odd $p$. We will demonstrate this fact by computing the total indicators.
By \cite[Ch. 4]{SerreACA}, there are exactly two isometric classes of non-degenerate symmetric bilinear form on an $r$-dimensional  $\f_p$-space. They are represented by $B_0$ and $B_1$ whose Gram matrix relative to the standard basis is given by
$$
\left[\begin{array}{c|c}
  I_{r-1}   & 0 \\ \hline
    0 & u
\end{array}\right]
$$
where $u$ can be any fixed quadratic nonresidue in $\f_p$, and $I_{r-1}$ denotes the identity matrix of rank $r-1$ . In particular, a non-degenerate symmetric bilinear form $B$ on $\f_p^r$ is determined by its \emph{discriminant} $\det(B)$ in $\f_p^\times/(\f_p^\times)^2$, where
$$
\det B  = \det\left([B(e_i, e_j)]_{ij}\right) \in \f_p^\times\,.
$$
 Therefore, the isometric class of a bilinear form $B$ is uniquely determined by the Legendre symbol $\Jac{\det B}{p}$ which is 1 if $\det B$ is a quadratic residue, and -1 otherwise.
 We can now compute the total indicators using the following formula obtained by Shimizu
\cite[Thm. 6.1]{Sh}: The $n$-th indicator of $m$ in $\TY(\f_p^r, \chi_B, \tau)$ for any non-degenerate symmetric bilinear form $B$ on $\f_p^r$ is given by
\begin{equation}\label{eq:5.2}
 \nu_n(m)=\left\{\begin{array}{ll}
   0 & \text{if  $n$ is odd},\medskip\\
   \sgn(\tau)^{k}\e_{p}^{r(k+1)}\Jac{-k}{p}^r \Jac{-2}{p}^{r(k+1)}
\Jac{\det B}{p}^{k+1} & \text{if }  n=2k \text{ and } p\nmid k,\medskip\\
\sgn(\tau)^{k}\e_{p}^{rk}\Jac{-2}{p}^{rk}
\Jac{\det B}{p}^{k} \sqrt{p^r} & \text{if }  n=2k \text{ and } p\mid k,\smallskip\\
\end{array}
 \right.
\end{equation}
where $\e_p=\sqrt{\Jac{-1}{p}}$.
\begin{proposition} \label{p:5.4}
Let $V=\f_p^{2\ell}$, $\chi_B$ the bicharacter associated with a non-degenerate bilinear form $B$ on $V$, $\tau$ a square root of $|V|\inv$, and $\CC=\mathcal{TY}(V, \chi_B,\tau)$.
\begin{enumerate}
  \item[(i)]  If $\Jac{-1}{p}^{\ell} \Jac{\det B}{p}  \sgn(\tau)=1$, then  $\FSexp(\CC) = 2p$ and $$\v(m)=p^{\ell}+(p-1) \sgn(\tau)\,. $$
  \item[(ii)] If $\Jac{-1}{p}^{\ell} \Jac{\det B}{p}\sgn(\tau)=-1$, then $\FSexp(\CC) = 4p$ and $\v(m)=0$.
\end{enumerate}
\end{proposition}
\begin{proof} By \cite[Thm. 3.2]{Sh}, $\nu_n(a) = \delta_{na, 0}$ for $a \in V$. Therefore, $p
\mid \FSexp(\CC)$. To determine $\FSexp(\CC)$, it is enough to consider
the values $\nu_n(m)$ with $n=2k$ and $p \mid k$ by virtue of \eqref{eq:5.2}. Note that
$$
\e_p^{2\ell} \Jac{\det B}{p} = \Jac{-1}{p}^{\ell}
\Jac{\det B}{p} = \ep \sgn(\tau)
$$
for some $\ep=\pm 1$.  Therefore, \eqref{eq:5.2} becomes
\begin{equation}\label{eq:5.3}
 \nu_n(m)=\left\{\begin{array}{ll}
   0 & \text{if  $n$ is odd},\smallskip\\
   \sgn(\tau) \ep^{k+1} & \text{if }  n=2k \text{ and } p\nmid k,\smallskip\\
\ep^{k} p^{\ell}& \text{if }  n=2k \text{ and } p\mid k,\smallskip\\
\end{array}
 \right.
\end{equation}
If $\ep=1$, then $\nu_{2p}(m)=p^{\ell} = d(m)$ and so $\FSexp(\CC) =2p$. Moreover,
$$
 \v(m) = p^{\ell} +\sum_{k=1}^{p-1} \nu_{2k}(m)
  = p^{\ell}+\sgn(\tau) (p-1)\,.
$$
If $\ep=-1$, then $\nu_{2p}(m)=-p^{\ell}$ and $\nu_{4p}(m)=p^{\ell}$. Thus, $\FSexp(\CC)=4p$ and
$$
 \v(m) = \sum_{k=1}^{2p} \nu_{2k}(m) =0\,. \qedhere
$$
\end{proof}

\begin{corollary}\label{c:5.5}
Let  $B$ be a non-degenerate symmetric bilinear form on $V=\f_p^{2\ell}$ and $\tau = \pm p^{-\ell}$.
Then $\CC=\TY(V, \chi_B, \tau)$ admits a fibration if, and only if, $\v(s) \ge d(s)$ for all simple object $s \in \CC$.
\end{corollary}
\begin{proof}
  By the preceding proposition, for $s \in V$, $\v(s) \ge 2 > d(s)$. In particular, the inequality  holds for all $s \in V$ automatically. Therefore, we only need to consider the simple object $m$.

  By \cite[Prop. 4.1]{Ta}, $\CC$ admits a fibration if, and only if, $\tau=p^{-\ell}$ and $B$ is hyperbolic, i.e. the Gram matrix of $B$ relative to some basis of $\f_p^{2\ell}$ is of the form
  $\mtx{0 & I_\ell \\ \hline I_\ell & 0}$, or equivalently, $\Jac{\det B}{p} = \Jac{-1}{p}^\ell$.

  If $\CC$ admits a fibration, then $\Jac{\det B}{p} \Jac{-1}{p}^\ell = 1 =\sgn(\tau)$ and so $\v(m) = p^\ell + p-1 >  p^\ell =d(m)$ by Proposition \ref{p:5.4}.
  Conversely, if $\v(m)\ge p^\ell$, then, by Proposition \ref{p:5.4}, $\sgn(\tau) = 1$ and $\Jac{\det B}{p} \Jac{-1}{p}^\ell = 1$, or equivalently $\Jac{\det B}{p}  = \Jac{-1}{p}^\ell$. Therefore, $\CC$ admits a fibration by the preceding paragraph.
\end{proof}

\begin{remark} {\rm The corollary implies that there exists a genuine semisimple quasi-Hopf algebra of which the total indicators of its representations are all positive. Let $B$ be the bilinear form on $\f_p^{2\ell}$ whose Gram matrix relative to the standard basis is $\mtx{I_{2\ell-1} & 0\\ \hline 0 & (-1)^{\ell+1}}$. Then $\Jac{\det B}{p} \Jac{-1}{p}^\ell=-1$. If we take $\sgn(\tau)=-1$ or $\tau=-p^\ell$, then, by Corollary \ref{c:5.5} and Proposition \ref{p:5.4},  $\TY(\f_p^{2\ell}, \chi_B, \tau)$ is monoidally equivalent to $\Rep(H)$ of a \emph{genuine} quasi-Hopf algebra $H$ with positive total indicators. Therefore, in general, the existence of vanishing total indicators is not a necessary condition for a semisimple quasi-Hopf algebra being genuine.
}
\end{remark}


\end{document}